\documentclass[11pt,reqno]{amsart}
\usepackage[T1]{fontenc}
\usepackage{lmodern,color}
\usepackage{amsmath,amssymb,amsthm}
\usepackage{microtype}
\usepackage{enumitem}
\usepackage[hidelinks]{hyperref}
\usepackage[left=1.25in,right=1.25in,top=1.1in,bottom=1.1in]{geometry}

\newtheorem{theorem}{Theorem}[section]
\newtheorem{proposition}[theorem]{Proposition}
\newtheorem{lemma}[theorem]{Lemma}

\newtheorem{conjecture}[theorem]{Conjecture}

\numberwithin{equation}{section}

\newcommand{\R}{\mathbb R}
\newcommand{\Sn}{\mathbb S^{n-1}}

\newcommand{\cK}{\mathcal K}
\newcommand{\V}{\operatorname{V}}
\DeclareMathOperator{\supp}{supp}
\DeclareMathOperator{\conv}{conv}

\title[The B\'ezout inequality characterizes simplices]
{The B\'ezout inequality for mixed volumes characterizes simplices}
\date{}

\author[D. Langharst]{Dylan Langharst}
\address{Carnegie Mellon University\\ Department of Mathematical Sciences \\ Pittsburgh, PA 15213, USA\\ ORCID: 0000-0002-4767-3371}
\email{dlanghar@andrew.cmu.edu}

\author[S. Wang]{Shouda Wang}
\address{KTH Royal Institute of Technology \\ Department of Mathematics \\  Stockholm, 11428, Sweden \\ ORCID: 0000-0002-6400-2262}
\email{shoudawang@princeton.edu}

\begin{document}

\begin{abstract}
We prove that the B\'ezout inequality for mixed volumes characterizes simplices among full dimensional convex bodies in every dimension, resolving a conjecture of Soprunov and Zvavitch. We give two separate proofs of the conjecture. Along the way, we also prove characterizations of simplices in terms of longest chords or relative inradii.

\end{abstract}

\maketitle

\section{Introduction}

Minkowski's second inequality states that, for any two convex bodies $K,L\subset\R^n$, one has
\begin{equation}
\label{eq:Min2}
\V_n(K)\,
\V_n(K[n-2],L,L)
\leqslant
\V_n(K[n-1],L)^2,
\end{equation}
where $\V_n$ is the mixed volume in $\R^n$.
This inequality yields many other useful geometric inequalities such as the Brunn-Minkowski inequality and the isoperimetric inequality.
Of particular interest is another consequence of \eqref{eq:Min2} called \emph{Fenchel's inequality}:  for every triple $A,B,K\subset\R^n$ of convex bodies, one has
\begin{equation}
\label{eq:fenchel}
\V_n(K)\,
\V_n(K[n-2],A,B)
\leqslant
2
\V_n(K[n-1],A)
\V_n(K[n-1],B).
\end{equation}
This inequality is very reminiscent of  the classical Bézout inequality  from algebraic geometry (see \cite{SZ16}), where the only difference is that the constant in the Bézout inequality is $1$ rather  than $2$ in \eqref{eq:fenchel}.
One naturally asks whether Fenchel's inequality also holds with a  constant $c$ smaller than $2$
\begin{equation}
\label{eq:c-fenchel}
\V_n(K)\,\V_n(K[n-2],A,B) \leqslant c\V_n(K[n-1],A)\V_n(K[n-1],B).
\end{equation}
It is known, however, that the constant $2$ in Fenchel's inequality is sharp in the class of general convex bodies. In fact, equality  is achieved when $K$ is a cross-polytope and $A,B$ are certain line segments \cite{SM24}. 
On the other hand,
there has been  much progress showing that \eqref{eq:c-fenchel} indeed holds for some constant $c\in (1,2)$ under  additional geometric assumptions \cite{AFO14,FMMZ24,AS26,FHMNWZ26}.

The case $c=1$, i.e., the Bézout inequality,  was studied by Soprunov and Zvavitch \cite{SZ16}, who proved it for  all convex bodies $A,B$ when $K$ is a simplex. 
They furthermore conjectured that simplices are in fact the only convex bodies $K$ for which the  Bézout inequality holds for arbitrary $A,B$.

\begin{conjecture}
\label{conj}
Let $n\geq 2$ and let $K\subset \R^n$ be an $n$-dimensional convex body. Assume that for any two convex bodies $A,B\subset \R^n$,
\begin{equation}\label{eq1}
\V_n(K)\,\V_n(A,B,K[n-2]) \leqslant \V_n(K[n-1],A)\V_n(K[n-1],B)
\end{equation}
holds. Then $K$ must be an $n$-dimensional simplex.
\end{conjecture}
Soprunov and Zvavitch \cite{SZ16} already showed Conjecture~\ref{conj} when $n=2$, and Saroglou, Soprunov and Zvavitch \cite{SSZ16} proved Conjecture~\ref{conj} when $K$ is a convex polytope in $\R^n$. 
Later, they  \cite{SSZ19} proved that a variant of \eqref{eq1} for the mixed volume of $n$ arbitrary bodies characterizes simplices, which turned out to settle Conjecture~\ref{conj} in $\R^3$.  
Some partial progress towards Conjecture~\ref{conj} can be found in \cite{SM24,SM26}.

In the present paper, we prove Conjecture~\ref{conj} for all $n\geq 2$ using two different methods. The first is inspired by Chakerian's \cite{Cha67} proof of the Rogers-Shephard inequality and  needs only the special test pairs in \textup{(iii)} below.
The second proof uses inner parallel bodies. 

We introduce some notation.  For $u\in \Sn$, we denote by $\ell_K(u)$ the length of the longest chord in $K$ parallel to $u$.
And the relative inradius of $K$ with respect to a full dimensional convex body $M$ is defined by
$$
r(K,M)=\max\{t\geqslant0:z+tM\subseteq K
\text{ for some }z\in\R^n\}.
$$

\begin{theorem}\label{thm1}
For an $n$-dimensional convex body $K\subset\R^n$, $n\geq 2$, the following are equivalent.
\begin{enumerate}[label=\textup{(\roman*)},leftmargin=*,itemsep=2pt]
\item $K$ is an $n$-dimensional simplex.
\item Inequality \eqref{eq1} holds for all $A,B\in\cK^n$.
\item For every nonzero $x\in\operatorname{int}(K+(-K))$,
inequality \eqref{eq1} holds with
$A=[0,x]$ and $B=K\cap(K-x)$.
\item For every direction $u\in \Sn$, 
\begin{equation}\label{eq18}
    \ell_K(u) = \frac{\V_n(K)}{\V_n(K[n-1],[0,u])}.
\end{equation}

\item For every full dimensional convex body $M$,
\begin{equation}\label{eq2}
r(K,M)=\frac{\V_n(K)}{\V_n(K[n-1],M)}.
\end{equation}
\end{enumerate}
\end{theorem}
Once this theorem is proven, the equivalence between \textup{(i)} and \textup{(ii)} immediately confirms Conjecture \ref{conj}.
We remark that the original conjecture in \cite{SZ16} was stated with a parameter $r$.
They conjectured that, if $K$ is a full dimensional convex body such that 
$$
\V_n(A_1,\dots, A_r,K[n-r]) \V_n(K)^{r-1}\leq 
\prod_{i=1}^r \V_n(A_i, K[n-1])
$$
holds for arbitrary convex bodies $A_1,\dots, A_r$, then $K$ has to be a simplex. Conjecture \ref{conj} corresponds to the $r=2$ case. However, the cases where  $r>2$ can be  reduced to the $r=2$ case by choosing $A_3=\cdots=A_r=K$. Therefore, our solution to Conjecture \ref{conj} indeed solves the original conjecture \cite[Conj.1.1]{SZ16}.

The proof of Theorem \ref{thm1} is organized as follows.
The implication \textup{(i)}$\Rightarrow$\textup{(ii)} is \cite[Thm.2.5]{SZ16}. The implication \textup{(ii)}$\Rightarrow$\textup{(iii)} is trivial. In Section~\ref{sec:covario} we prove \textup{(iii)}$\Rightarrow$\textup{(iv)}$\Rightarrow$\textup{(i)}, and Section~\ref{sec:inner} proves \textup{(ii)}$\Rightarrow$\textup{(v)}$\Rightarrow$\textup{(i)}. Therefore, Sections \ref{sec:covario} and \ref{sec:inner} provide two separate proofs of Conjecture \ref{conj}.

At a high level,   \textup{(ii)}$\Rightarrow$\textup{(iv)} in Section \ref{sec:covario} and \textup{(ii)}$\Rightarrow$\textup{(v)} in Section \ref{sec:inner}   follow a common strategy.
We first express $\V_n(K)$ as an integral, using the layer cake
formula in the first argument and Alexandrov's variational formula
in the second.
We then combine \eqref{eq1} with Minkowski's first inequality / Alexandrov-Fenchel inequality
to bound the integrand.
It turns out that the integral of this upper bound is itself at most $\V_n(K)$,
which forces equality almost everywhere.
In particular, this yields \textup{(iv)} and \textup{(v)},
respectively.

The implication 
\textup{(iv)}$\Rightarrow$\textup{(i)} is proved using the characterization of simplices by the covariogram function. And the implication 
\textup{(v)}$\Rightarrow$\textup{(i)} is proved by first reducing to polytopes and then using the integral representation of mixed volumes by surface area measures.

\section{Preliminaries}
\label{sec:prelim}
We use standard facts from Schneider's book \cite{Sh1}. Let $n\geq2$, and let $\cK^n$ denote the class of convex bodies (non-empty compact convex sets) in $\R^n$. We write $\cK^n_n$ for the subclass of full dimensional convex bodies. Denote by $\Sn$  the unit Euclidean sphere.

For $K\in\cK^n$, write $h_K(u)=\max_{x\in K}\langle x,u\rangle$ for its support function. For $u\in\Sn$ and $t\in\R$, let $H_{u,t}=\{x\in\R^n:\langle x,u\rangle=t\}$ and $H^-_{u,t}=\{x\in\R^n:\langle x,u\rangle\leq t\}$. The exposed face of $K$ with outer normal $u$ is $F(K,u)=K\cap H_{u,h_K(u)}$. For any linear subspace $E$ of $\R^n$, denote by $P_{E}$ the  orthogonal projection onto $E$.

\subsection{Volume and mixed volume}\label{subsec:mixed-volumes}
The Minkowski sum of nonempty compact convex sets $K,L\subset\R^n$ is $K+L=\{x+y:x\in K,y\in L\}$. The Hausdorff distance between $K,L\in \cK^n$ is given by
$\delta_H(K,L)=\max\left\{\sup_{x\in K}\inf_{y\in L}|x-y|,\,\sup_{y\in L}\inf_{x\in K}|x-y|\right\}$.
We view $\cK^n$ as a metric space equipped with the Hausdorff metric $\delta_H$. Volume and mixed volume are continuous with respect to the Hausdorff metric.

We write $\V_n(K_1,\ldots,K_n)$ for the mixed volume of $K_1,\dots,K_{n}\in \cK^n$. Repeated arguments in mixed volumes and mixed area measures are indicated by brackets, that is,  $K[k]$ denotes $k$ copies of $K$. For example, $\V_n(K[n])=\V_n(K,\dots, K)$ is the volume of $K$, which will also be denoted by $\V_n(K)$.

Mixed volumes are symmetric in their arguments and multilinear \cite[eq.(5.26)]{Sh1}, i.e., for $K_1,\dots,K_{n-1},L,M\in \cK^n$ and $ a,t \ge 0$,
\begin{equation}
\label{eq:multi_linear}
\V_n(K_1,\ldots,K_{n-1},aL+tM) = a\V_n(K_1,\ldots,K_{n-1},L)+t\V_n(K_1,\ldots,K_{n-1},M).
\end{equation}
 The mixed volumes are also monotone, that is, if $L\subseteq M$, then
\begin{equation}
    \label{eq:mono}
    \V_n(K[n-1],L) \leq \V_n(K[n-1],M).
\end{equation}
For $v\in\Sn$, we have the projection formula, \cite[Thm.5.3.1]{Sh1},
\begin{equation}\label{eq9_2}
\V_{n-1}(P_{v^\perp}K_1,\dots, P_{v^\perp}K_{n-1}) = n\V_n(K_1,\dots, K_{n-1},[0,v]), \qquad K_i\in \cK^n
\end{equation}
Mixed volume can be represented by integration with respect to the mixed area measure $S_{K_1,\dots,{K_{n-1}}}$ of $K_1,\dots,K_{n-1}\in \cK^n$ (see  \cite[Thm.5.1.7]{Sh1}):
\begin{equation}\label{eq:mixed_volume_diff}
\V_n(K_1,\ldots,K_{n})
=\frac{1}{n}\int_{\Sn}h_{K_n}(u)\,dS_{K_1,\ldots,K_{n-1}}(u), \qquad K_i\in \cK^n.
\end{equation}
We write $S_K=S_{K[n-1]}$ for the surface area measure of $K$.
For a polytope $P\in\cK^n_n$, we have $S_P=\sum_F\V_{n-1}(F)\delta_{u_F}$, where the sum runs over its facets $F$ and $u_F$ is the outer unit normal of $F$. For $K\in\cK^n_n$, the support of $S_K$ is not contained in any closed hemisphere. In particular, it contains finitely many directions whose convex hull has the origin in its interior. 

Finally, we recall two  classical inequalities about mixed volumes.
One is  \textit{Minkowski's first inequality} \cite[Thm.7.2.1]{Sh1}:
\begin{equation}
    \V_n(K[n-1],L)^n \geq \V_n(K)^{n-1}\V_n(L).
    \label{eq:min_first}
\end{equation}
The other is the following consequence of the Alexandrov-Fenchel inequality
(see \cite[eq.(7.64)]{Sh1}):
\begin{equation}
    \V_n(K[n-2],L,M)^{n-1} \geq \V_n(K[n-1],M)^{n-2}\V_n(L[n-1],M).
    \label{eq:AF}
\end{equation}

\subsection{Inner parallel bodies}\label{subsec:inner-parallel}
The Wulff shape of a continuous function $h\in C(\Sn)$ is the convex subset of $\R^n$ given by the intersection of halfspaces $[h]=\bigcap_{u\in\Sn}H^-_{u,h(u)}$.
It follows that, if $[h]\neq \emptyset$, then $h_{[h]} \leq h$ holds pointwise. Moreover, for $K\in \cK^n$, $[h_K]=K$.

We recall Alexandrov's variational formula \cite[Lem.7.5.3]{Sh1}.
\begin{lemma}[Alexandrov's variational formula]\label{lem:alexandrov}
Let $K\in \cK^n_n$  and let $f\in C(\Sn)$. Then
\begin{equation*}
\lim_{t\to0}\frac{\V_n([h_K+tf])-\V_n(K)}{t}
=\int_{\Sn}f(u)\,dS_K(u).
\end{equation*}
\end{lemma}

For $K,M\in\cK^n$, their \textit{Minkowski difference}
(see \cite[p.~146]{Sh1}) is $$K\ominus M=\{x\in\R^n:x+M\subseteq K\}=[h_K-h_M].$$
Recall that the \emph{relative inradius} $r(K,M)$ defined in the introduction
is the largest $r\geq0$ for which $K\ominus rM$ is nonempty.
For  $0\leq t\leq r(K,M)$, the convex body $K_t=K\ominus tM$ is called the inner parallel body of $K$ relative to $M$ at distance $t$.

Choose $z\in K\ominus r(K,M)M$.
Monotonicity and linearity of mixed volumes give
\[
\V_n(K)
\geq \V_n(K[n-1],z+r(K,M)M)
=r(K,M)\V_n(K[n-1],M).
\]
Thus, for any $K,M\in \cK^n_n$,
\begin{equation}
\label{eq:inradius_inequality}
r(K,M)\leq\frac{\V_n(K)}{\V_n(K[n-1],M)}.
\end{equation}
That is, \textup{(v)} in Theorem \ref{thm1}  holds with inequality for any $K$ and $M$. And the point of  Theorem \ref{thm1} is that \eqref{eq:inradius_inequality} holds with equality for all $M$ if and only if $K$ is a simplex.

 For completeness, we record some well-known  properties of inner parallel bodies.

\begin{proposition}\label{p:continuity}
Let $K,M\in\cK^n_n$, $r=r(K,M)$, and $K_t=K\ominus tM$. Then $K_t\in\cK^n_n$ for $0\leq t<r$, the family is  continuous in Hausdorff distance on $[0,r]$, and $\V_n(K_r)=0$.
\end{proposition}

\begin{proof}
By translating $M$, we may assume that $0\in\operatorname{int}(M)$. It is clear from the definition that each $K_t$ is compact and convex. Choose $z\in K_r$. Then $z+(r-t)M\subseteq K_t$ for $t\leq r$, hence $K_t$ has nonempty interior for $t<r$. 

Now we show continuity in Hausdorff distance. Fix $0\leq t<s<r$ and set $\lambda=\frac{s-t}{r-t}$.
Since $K_s\subseteq K_t$, it suffices to estimate the distance from points of $K_t$ to $K_s$. 
Fix $x\in K_t$ and set $y=(1-\lambda)x+\lambda z$.
For every $m\in M$, $y+sm=(1-\lambda)(x+tm)+\lambda(z+rm)\in K$, and hence $y\in K_s$. Therefore, 
\begin{equation}
\nonumber
\delta_H(K_t,K_s)
\leq\frac{s-t}{r-t}\operatorname{diam}(K),
\qquad 0\leq t<s<r,
\end{equation}
where $\operatorname{diam}(K)=\sup_{x,y\in K}|x-y|$ is the diameter of $K$. We have thus proven the continuity in Hausdorff distance on $[0,r)$.
As $t\uparrow r$,  the sets $K_t$ decrease to $K_r$. Indeed, $K_r=\bigcap_{0\leq t<r}K_t$. Thus $t\mapsto K_t$ is Hausdorff continuous on $[0,r]$.

Finally, if $\V_n(K_r)>0$, then $K_r$ has nonempty interior. Since $0\in\operatorname{int}(M)$, there exist $x\in\R^n$ and $\varepsilon>0$ such that $x+\varepsilon M\subseteq K_r$. As $K_r+rM\subseteq K$, we would then have $x+(r+\varepsilon)M\subseteq K$, contradicting the definition of $r$. Hence $\V_n(K_r)=0$.
\end{proof}

\begin{lemma}\label{lem:inner-volume}
Let $K,M\in \cK^n_n$, and write $K_t=K\ominus tM$. Then for all $ 0 <t<r(K,M)$, we have
\begin{equation}\label{eq5}
-\frac{d}{dt^+}\V_n(K_t)=n\V_n(K_t[n-1],M),
\end{equation}
where $\frac{d}{dt^+}$ denotes the right derivative.
Moreover,
\begin{equation}\label{eq7}
\V_n(K)=n\int_0^{r(K,M)}\V_n(K_s[n-1],M)\,ds.
\end{equation}
\end{lemma}
\begin{proof}
Denote $r=r(K,M)$.
Fix $0< t<r$. For  $s\in (0, r-t)$, we have
\[
K_{t+s}=K\ominus(t+s)M=(K\ominus tM)\ominus sM
=K_t\ominus sM=[h_{K_t}-sh_M].
\]
By Proposition~\ref{p:continuity}, $K_{t+s}$ is a convex body. Apply Lemma~\ref{lem:alexandrov} with $f=-h_M$, and then use \eqref{eq:mixed_volume_diff}, we obtain
\[
-\frac{d}{ds}\bigg|_{s=0^+}\V_n(K_{t+s})
=\int_{\Sn}h_M(u)\,dS_{K_t}(u)
=n\V_n(K_t[n-1],M),
\]
which proves \eqref{eq5}. By Proposition~\ref{p:continuity} and continuity of mixed volumes, the right-hand side of \eqref{eq5} is continuous on $[0,r]$. Integrating \eqref{eq5} and using $\V_n(K_r)=0$ yield \eqref{eq7}.
\end{proof}

\subsection{The covariogram}\label{subsec:covariogram}
In this subsection we recall   Chakerian's proof \cite{Cha67} of the Rogers--Shephard inequality \cite{RS57}, which directly motivates the proof in Section \ref{sec:covario}. Section \ref{sec:inner} also follows a similar strategy.

Fix a convex body $C$ with $o\in \operatorname{int}C$,
 the  \textit{Minkowski functional} of $C$  is given by $\|x\|_C=\inf\{t>0:x\in tC\}$, $x\in\R^n$.
For  $K\in \cK^n_n$,
the \textit{difference body} of $K$ is
$DK=K+(-K)=\{x\in\R^n:K\cap(K+x)\neq\emptyset\}$.
The Rogers--Shephard inequality states that
\begin{equation}
\label{eq:RS}
\V_n(DK)\leq\binom{2n}{n}\V_n(K).
\end{equation}
Equality holds in \eqref{eq:RS} for $K\in \cK^n_n$ if and only if $K$ is an $n$-dimensional simplex. Chakerian gave a proof of \eqref{eq:RS} based on the \textit{covariogram} of $K$, defined by
\begin{equation}
\label{eq:covario}
g_K(x)=\V_n(K\cap (K+x)), \qquad x\in \R^n.
\end{equation}
We recommend the reader consult Bianchi's survey \cite{GB23} for a thorough overview of the covariogram function. We  recall a few facts about the covariogram function:
\begin{enumerate}
    \item $x\mapsto g_K(x)^{1/n}$  is supported on $DK$ and is concave in $DK$;
    \item Since $g_K = 0$ on $\partial (DK)$ and  $g_K(0)=\V_n(K)$, it follows from (1) that for all $x\in DK$,
    \begin{equation}
    \label{eq:covario_support}
    g_K(x) \geq \V_n(K)(1-\|x\|_{DK})^n.
    \end{equation}
    \item Finally, $\V_n(K)^2=\int_{\R^n}g_K(x)\,dx$.
\end{enumerate}
The covariogram function can be used to characterize simplices \cite{RS57}; see also \cite[Prop.4.7]{HLPRY25}. Chakerian used this inequality to characterize the equality in \eqref{eq:RS}.

\begin{proposition}\label{p:RS}
A convex body $K\in \cK^n_n$ is a simplex if and only if, for every $u\in\Sn$, the function $s\mapsto g_K(su)^{1/n}$ is affine on $[0,\|u\|_{DK}^{-1}]$.
\end{proposition}

Finally, for a convex body $C$ with $0\in\operatorname{int}C$,
we have by polar integration that
 \begin{equation}
 \label{eq:volume_formula}
 \V_n(C)=\frac{1}{n}\int_{\Sn} \|\theta\|_C^{-n}d\theta,
 \end{equation}
 where $d\theta$ denotes the spherical Lebesgue measure.

 \begin{proof}[Chakerian's proof of \eqref{eq:RS} and equality characterization]
 By polar coordinate integration, \eqref{eq:covario_support},  and \eqref{eq:volume_formula}, we obtain
 \begin{align*}
     \V_n(K) &= \int_{\R^n} \frac{g_K(x) dx}{\V_n(K)} = \int_{\Sn}\int_0^{\|\theta\|_{DK}^{-1}}\frac{g_K(r\theta) r^{n-1}dr}{\V_n(K)} \, d\theta
     \\
     & \geq \int_{\Sn}\int_0^{\|\theta\|_{DK}^{-1}}(1-r\|\theta\|_{DK})^nr^{n-1} dr d\theta
     \\
     & = \int_{\Sn}\|\theta\|_{DK}^{-n}\int_0^{1}(1-r)^nr^{n-1} dr d\theta
     \\
     &= \V_n(DK) \left(n \int_0^1 (1-r)^nr^{n-1}dr\right).
 \end{align*}
 Finally, $\left(n \int_0^1 (1-r)^nr^{n-1}dr\right)=\binom{2n}{n}^{-1}$ via the Beta function. This proves the inequality. 
 
 When equality $ \binom{2n}{n}\V_n(K) = \V_n(DK)$ holds, the inequality \eqref{eq:covario_support} has to achieve equality for every $r$ and $\theta$. Hence $g_K(r\theta)^{\frac{1}{n}} = \V_n(K)^{\frac{1}{n}}(1-\|x\|_{DK})$ is an affine function on $[0,\|u\|_{DK}^{-1}]$ for every fixed $u\in \Sn$. Proposition~\ref{p:RS} now implies that $K$ is a simplex.
 \end{proof}

\section{Parallel chords and the covariogram}\label{sec:covario}
This section establishes the implications \textup{(iii)}$\Rightarrow$\textup{(iv)}$\Rightarrow$\textup{(i)} in Theorem~\ref{thm1}.

\subsection{Chord lengths and the layer cake formula}\label{subsec:chord-facts}
We first recall an identity for the volume of a convex body which will be used later.
For a convex body $K$ and a fixed direction $v\in \Sn$, we define $\ell_K(v)$ as the length of the longest chord of $K$ parallel to $v$. For $s\geq 0$, let $K^s=K\cap(K-sv)$. Then $\ell_K(v)$ is the largest $s\geq0$ for which $K^s$ is nonempty.
\begin{lemma}\label{lem3.1}
    Let $K\in \cK^n$. Then for any $v\in \Sn$ and any $t\in (0, \ell_K(v))$, we have
    \begin{equation}
    \nonumber
        g_K(-tv) = \V_n(K^t)  =  \int_{t}^{\ell_K(v)} \V_{n-1}(P_{v^\perp} K^s)\, ds.
    \end{equation}
\end{lemma}
\begin{proof}For any $y\in P_{v^\perp}K$, define 
\[
\ell_K(v;y)=\V_1(K\cap(y+\R v)), 
\]
which is the length of the chord $K\cap (y+\R v)$. 
Then we have $\ell_K(v)=\max_{y\in P_{v^\perp}K}\ell_K(v;y)$.
For $y\in P_{v^\perp}K$, the chord $K\cap(y+\R v)$ is a compact interval $I_y$ of length $\ell_K(v;y)$. Therefore the fiber $K^s\cap(y+\R v)$ corresponds to the intersection $I_y\cap(I_y-s)$, and hence
\begin{equation}\nonumber
\V_1(K^s\cap(y+\R v))=(\ell_K(v;y)-s)_+,
\qquad s\geq0.
\end{equation}
 In particular, the fiber $K^s\cap(y+\R v)$ is nonempty if and only if $\ell_K(v;y)\geq s$. Hence,
\[
P_{v^\perp}K^s
=
\{y\in P_{v^\perp}K:\ell_K(v;y)\geq s\},
\qquad s\geq0.
\]
Denote by $\chi_A$  the characteristic function of a set $A$, i.e. $\chi_A(x)=1$ for $x\in A$ and $\chi_A(x)=0$ for $x\notin A$. By Fubini's theorem, for $0\leq t\leq{\ell_K(v)}$ we have
\begin{align*}
\nonumber
g_K(-tv)&
=\V_n(K^t)= \int_{K^t} \chi_{K^t}(x)\,dx=\int_{P_{v^\perp}K^t} \int_{\R}\chi_{K^t}(y+sv)\,ds\,dy 
\\
&=\int_{P_{v^\perp}K^t} \V_1(K^t\cap(y+\R v))\,dy\\
&=\int_{P_{v^\perp}K}(\ell_K(v;y)-t)_+\,dy.
\end{align*}
Note that we may trivially write 
$(\ell_K(v;y)-t)_+ = \int_{t}^{\ell_K(v)} \chi_{\{s\leq \ell_K(v;y)\}}\,ds$.
Therefore, for $0\leq t\leq{\ell_K(v)}$ we have
\begin{equation}\label{eq11}
\begin{aligned}
g_K(-tv)&=\int_{P_{v^\perp}K}\int_{t}^{\ell_K(v)} \chi_{\{s\leq \ell_K(v;y)\}}\,ds\,dy=\int_t^{\ell_K(v)}\V_{n-1}(P_{v^\perp}K^s)\,ds.
\end{aligned}
\end{equation}
In particular,
\begin{equation}\label{eq:layercake}
\V_n(K)=\int_0^{\ell_K(v)}\V_{n-1}(P_{v^\perp}K^s)\,ds.
\end{equation}
We conclude.
\end{proof}

We also record two useful facts about $K^s=K\cap (K-sv)$ for full dimensional $K$. Denote $\ell=\ell_K(v)$.
First, we claim that for $0\leq s\leq\ell$, we have
\begin{equation}
\label{eq:M_inclusions}
K^s+[0,sv]\subseteq K.
\end{equation}
Indeed, if $x\in K^s$, then $x,x+sv\in K$, and hence the entire segment $x+[0,sv]=[x,x+sv]$ is contained in $K$. 

Second, let $a,a+\ell v\in K$ be the endpoints of a longest chord parallel to $v$. For $0\leq s\leq\ell$, convexity of $K$ readily implies
$\left(1-\frac{s}{\ell}\right)K+\frac{s}{\ell}a\subseteq K$.
Moreover,
\[
\left[
\left(1-\frac{s}{\ell}\right)K+\frac{s}{\ell}a
\right]+sv
=
\left(1-\frac{s}{\ell}\right)K
+\frac{s}{\ell}(a+\ell v)
\subseteq K.
\]
Therefore, for $0\leq s\leq\ell$ we have
\begin{equation}\label{eq:chord_homothety}
\left(1-\frac{s}{\ell}\right)K+\frac{s}{\ell}a\subseteq K^s.
\end{equation}
In particular, $K^s$ has nonempty interior for $0\leq s<\ell$.

\subsection{Proof of \texorpdfstring{$\textup{(iii)}\Rightarrow\textup{(iv)}\Rightarrow\textup{(i)}$}{(iii) ⇒ (iv) ⇒ (i)} in Theorem~\ref{thm1}}\label{subsec:chord-bezout}

We first show $\textup{(iii)}\Rightarrow\textup{(iv)}$. 
 The proof uses Lemma \ref{lem3.1} to express $\V_n(K)=\V_n(K^0)$ as an integral of $\V_{n-1}(P_{v^\perp} K^s)$, and then upper bounds $\V_{n-1}(P_{v^\perp} K^s)$ using Minkowski's first inequality and the assumption \textup{(iii)}. The integral of the upper bound turns out to be equal to $\V_n(K)$, hence equality must be achieved almost everywhere, which in particular implies \textup{(iv)}.
\begin{proof}[Proof of $\textup{(iii)}\Rightarrow\textup{(iv)}$ in Theorem~\ref{thm1}]

We assume that $K\in \cK^n_n$ satisfies \textup{(iii)}. That is, for every $v\in\Sn$ and $0<s<\ell_K(v)$, \eqref{eq1} holds for $K$ with $A=[0,sv]$ and $B=K\cap(K-sv)$. 
We are going to prove that \textup{(iv)} holds, which is $\ell_K(u)=\frac{\V_n(K)}{\V_n(K[n-1],[0,u])}$.

Fix $v\in\Sn$, and write $K^s=K\cap(K-sv)$ and $\ell=\ell_K(v)$.
We first show an upper bound on $\V_n(P_{v^\perp}K^s)$. 
For $0<s<\ell$, the assumed B\'ezout inequality \eqref{eq1} and linearity \eqref{eq:multi_linear} yield
\begin{align*}
\V_n(K)\V_n(K[n-2],K^s,[0,v])\leq
\V_n(K[n-1],K^s)\V_n(K[n-1],[0,v]).
\end{align*}
On the other hand, the inclusion \eqref{eq:M_inclusions} and the monotonicity \eqref{eq:mono} and multilinearity \eqref{eq:multi_linear} of mixed volumes yield
\[
\V_n(K[n-1],K^s)
+s\V_n(K[n-1],[0,v])
\leq \V_n(K).
\]
Combining the two, for all $s\in (0,\ell)$ we obtain
\begin{align*}
\V_n(K[n-2],K^s,[0,v])
\leq
\V_n(K[n-1],[0,v])
\left(
1-\frac{\V_n(K[n-1],[0,v])}{\V_n(K)}s
\right).
\end{align*}
By the projection formula \eqref{eq9_2}, we have 
\begin{align}
\V_{n-1}((P_{v^\perp}K)[n-2],P_{v^\perp}K^s)
\leq
\V_{n-1}(P_{v^\perp}K)
\left(
1-\frac{\V_{n-1}(P_{v^\perp}K)}{n\V_n(K)}s
\right).
\label{eq30}
\end{align}
In particular, $1-\frac{\V_{n-1}(P_{v^\perp}K)}{n\V_n(K)}s\geq 0$ for $0<s<\ell$. Hence 
\begin{equation}
\nonumber
\bar\ell:=
\frac{\ell\,\V_{n-1}(P_{v^\perp}K)}{n\V_n(K)}
\leq1.
\end{equation}
Now by Minkowski's first inequality \eqref{eq:min_first} in $v^\perp$, we have
\[
\V_{n-1}\big((P_{v^\perp}K)[n-2],P_{v^\perp}K^s\big)^{n-1}
\geq
\V_{n-1}(P_{v^\perp}K)^{n-2}
\V_{n-1}(P_{v^\perp}K^s).
\]
Note this is merely equality when $n=2$. Combining this with \eqref{eq30}, we obtain
\begin{equation}\label{eq12}
\V_{n-1}(P_{v^\perp}K^s)
\leq
\V_{n-1}(P_{v^\perp}K)
\left(
1-\frac{\V_{n-1}(P_{v^\perp}K)}{n\V_n(K)}s
\right)^{n-1},
\qquad 0<s<\ell.
\end{equation}

Now that we have obtained an upper bound on $\V_{n-1}(P_{v^\perp}K^s)$, we may proceed to upper bounding $\V_n(K)$ using Lemma \ref{lem3.1}.
Now Lemma~\ref{lem3.1}, \eqref{eq12}, a change of variables, and $\bar \ell\leq 1$ yield
\begin{align*}
\V_n(K)
&=\int_0^\ell\V_{n-1}(P_{v^\perp}K^s)\,ds\\
&\leq\V_{n-1}(P_{v^\perp}K)
\int_0^\ell
\left(
1-\frac{\V_{n-1}(P_{v^\perp}K)}{n\V_n(K)}s
\right)^{n-1}\,ds\\
&=n\V_n(K)\int_0^{\bar\ell}(1-s)^{n-1}\,ds
\leq\V_n(K).
\end{align*}
All inequalities are therefore equalities. In particular, equality in the last step forces $\bar\ell=1$. Thus
\[
\ell_K(v)=\frac{n\V_n(K)}{\V_{n-1}(P_{v^\perp}K)}
=\frac{\V_n(K)}{\V_n(K[n-1],[0,v])},
\]
which completes the proof.
\end{proof}

We now prove the implication $\textup{(iv)}\Rightarrow\textup{(i)}$ using the characterization of simplices by the covariogram function  from Proposition \ref{p:RS}.
\begin{proof}[Proof of $\textup{(iv)}\Rightarrow\textup{(i)}$ in Theorem~\ref{thm1}]
Suppose that \textup{(iv)} holds. Fix $v\in\Sn$, write $\ell=\ell_K(v)$, and set $K^s=K\cap(K-sv)$.
Projecting the inclusion \eqref{eq:chord_homothety} onto $v^\perp$ yields
\begin{equation}
\V_{n-1}(P_{v^\perp}K^s)\geq\V_{n-1}(P_{v^\perp}K)
\left(1-\frac{s}{\ell}\right)^{n-1},
\qquad 0\leq s\leq\ell.
\label{eq20}
\end{equation}
Together with \eqref{eq:layercake}, this gives
\begin{align*}
\V_n(K)
&=\int_0^\ell\V_{n-1}(P_{v^\perp}K^s)\,ds\\
&\geq\V_{n-1}(P_{v^\perp}K)
\int_0^\ell\left(1-\frac{s}{\ell}\right)^{n-1}\,ds\\
&=\frac{\ell}{n}\V_{n-1}(P_{v^\perp}K)
=\V_n(K),
\end{align*}
where the last equality follows from \textup{(iv)} and \eqref{eq9_2} with  $K_1=\cdots=K_{n-1}=K$. Thus \eqref{eq20} is an equality for almost every $s\in(0,\ell)$. Since $g_K$ is continuous, there is equality everywhere. Substituting into \eqref{eq11} gives
\[
g_K(-sv)=\V_n(K)\left(1-\frac{s}{\ell}\right)^n,
\qquad 0\leq s\leq\ell.
\]
Since $v$ was arbitrary, Proposition~\ref{p:RS} implies that $K$ is a simplex.
\end{proof}

\section{Inner parallel bodies}\label{sec:inner}
This section proves the implications \textup{(ii)}$\Rightarrow$\textup{(v)}$\Rightarrow$\textup{(i)} in Theorem~\ref{thm1}.
The required facts about inner parallel bodies were established in Section~\ref{subsec:inner-parallel}. We first show that if $K$ satisfies the B\'ezout inequality, then the inequality \eqref{eq:inradius_inequality} becomes an equality for every $M$. That is, we prove that \textup{(ii)} implies \textup{(v)} in Theorem~\ref{thm1}.

\begin{proof}[Proof that $\textup{(ii)}\Rightarrow\textup{(v)}$ in Theorem~\ref{thm1}]
We prove that if $K\in \cK^n_n$ satisfies \eqref{eq1} for every $A,B\in\cK^n$,
then \eqref{eq2} holds for every  $M\in\cK^n_n$.

Fix  $M\in \cK^n_n$ and write $r=r(K,M)$. By translation invariance, we may assume that $0\in\operatorname{int}(M)$. For $0\leq t\leq r$, set $K_t=K\ominus tM$. Proposition~\ref{p:continuity} shows that $K_t$ is a convex body for $t<r$, and Lemma~\ref{lem:inner-volume} gives the integral identity \eqref{eq7}.

For $0\leq s<r$, using the iterated form of the Alexandrov-Fenchel inequality \eqref{eq:AF}, followed by \eqref{eq1}, we obtain
\begin{align}
\nonumber
\V_n(K_s[n-1],M)
&\leq
\frac{\V_n(K[n-2],K_s,M)^{n-1}}
{\V_n(K[n-1],M)^{n-2}}\\
\nonumber
&\leq
\frac{\V_n(K[n-1],M)}{\V_n(K)^{n-1}}
\V_n(K[n-1],K_s)^{n-1}\\
\label{eq13}
&\leq
\frac{\V_n(K[n-1],M)}{\V_n(K)^{n-1}}
\bigl(\V_n(K)-s\V_n(K[n-1],M)\bigr)^{n-1}.
\end{align}
The last inequality follows from $K_s+sM\subseteq K$, by monotonicity
and multilinearity of mixed volumes.

Set $\bar r=r\V_n(K[n-1],M)/\V_n(K)\leq1$.
Combining \eqref{eq13} with \eqref{eq7} and making the substitution
$t=s\V_n(K[n-1],M)/\V_n(K)$ gives
\begin{align}
\nonumber
\V_n(K)
&\leq n\int_0^r
\frac{\V_n(K[n-1],M)}{\V_n(K)^{n-1}}
\bigl(\V_n(K)-s\V_n(K[n-1],M)\bigr)^{n-1}\,ds\\
\label{eq6}
&=n\V_n(K)\int_0^{\bar r}(1-t)^{n-1}\,dt
\leq\V_n(K).
\end{align}
All inequalities in \eqref{eq6} are therefore equalities.
Since $(1-t)^{n-1}>0$ for $0\leq t<1$, equality in the last step
forces $\bar r=1$. Thus
$r(K,M)=\V_n(K)/\V_n(K[n-1],M)$, as desired.
\end{proof}

We now proceed to the proof that \textup{(v)} implies \textup{(i)} in Theorem~\ref{thm1}. 
We will first show this for polytopes, and then show that condition \textup{(v)} in fact always implies $K$ is a polytope.

\begin{lemma}\label{lem1}
Let $K\subset\R^n$ be a full-dimensional polytope. Assume that $r(K,M)\V_n(M,K[n-1])=\V_n(K)$ holds for every full-dimensional polytope $M$. Then $K$ is a simplex.
\end{lemma}

\begin{proof}
Translate $K$ so that $0\in\operatorname{int}K$, and write its (irredundant) facet representation as
$$
K=\bigcap_{i=1}^m\{x:\langle u_i,x\rangle\le h_K(u_i)\},
$$
where $u_i$ are unit outer normals of facets. By Carath\'eodory's theorem, there exists a set $I\subseteq\{1,\ldots,m\}$ with $|I|\le n+1$ and positive weights $\lambda_i$ such that $\sum_{i\in I}\lambda_i u_i=0$.

Assume $m>n+1$. Then we may choose $j\notin I$. Consider
$$
M=K\cap H^-_{u_j, h_K(u_j)-\varepsilon}.
$$
For sufficiently small $\varepsilon>0$, this polytope is full-dimensional and $h_K(u_i) = h_M(u_i)$ for every $i\neq j$.

We claim $r(K,M)=1$. Indeed, since $M\subseteq K$, we have $r(K,M)\ge1$. Conversely, assume $x+tM\subseteq K$ for some $x\in\R^n$ and $t\geq0$. Then $$\langle u_i,x\rangle+t h_K(u_i)\le h_K(u_i), \qquad  i\in I.$$ Multiplying both sides by $\lambda_i$ and summing over $i\in I$ yields
$$t\sum_{i\in I}\lambda_i h_K(u_i)
\le \sum_{i\in I}\lambda_i h_K(u_i),$$
where we used $\sum_{i\in I} \lambda_i u_i=0$. Since $\lambda_i h_K(u_i)>0$, we have $t\le1$ and therefore $r(K,M)\leq 1$.
Combining this with the lower bound, we obtain $r(K,M)=1$ as desired. 

However, we also have
$$
\begin{aligned}
\V_n(M,K[n-1])
&=\frac1n\sum_{i=1}^m \V_{n-1}(F(K,u_i))h_M(u_i) =\V_n(K)-\frac{\varepsilon}{n}\V_{n-1}(F(K,u_j))
<\V_n(K).
\end{aligned}
$$
This contradicts $\V_n(K) = r(K,M)\V_n(K[n-1],M)=\V_n(K[n-1],M)$.
Therefore, $m=n+1$. That is, $K$ has $n+1$ facets and hence is a simplex.
\end{proof}

We now show that \textup{(v)} forces $K$ to be a polytope.

\begin{lemma}\label{lem3}
     Let $K\in\cK^n_n$ and assume that $r(K,M) = \frac{\V_n(K)}{\V_n(M,K[n-1])}$ holds for every  $M\in\cK^n_n$.
    Then $K$ is a polytope.
\end{lemma}
\begin{proof}
 Since $K$ has nonempty interior, we may choose finitely many directions $v_1,\ldots,v_m\in\supp S_K$ such that
$0\in\operatorname{int}\conv\{v_1,\ldots,v_m\}$.
Then
\[
P=\bigcap_{i=1}^m H^-_{v_i,h_K(v_i)}
\]
is a bounded polytope containing $K$. Note that $h_P(v_i)=h_K(v_i)$ for every $i$.

Choose $a\in\R^n$ such that
$Q=a+r(K,P)P\subseteq K$. 
Then, by assumption and \eqref{eq:mixed_volume_diff}, we have
\[
\int_{\Sn}(h_K-h_Q)\,dS_K
=n\bigl(\V_n(K)-r(K,P)\V_n(K[n-1],P)\bigr)=0.
\]
Since $h_K-h_Q$ is continuous and nonnegative, it vanishes on $\supp S_K$. In particular, $h_Q(v_i)=h_K(v_i)=h_P(v_i)$ for every $i=1,\ldots,m$.
Since $Q=a+r(K,P)P$ and
\[
P=\bigcap_{i=1}^m H^-_{v_i,h_P(v_i)},
\]
we have
\[
Q
=
\bigcap_{i=1}^m
H^-_{v_i,\langle a,v_i\rangle+r(K,P)h_P(v_i)}
=
\bigcap_{i=1}^m H^-_{v_i,h_Q(v_i)}
=
\bigcap_{i=1}^m H^-_{v_i,h_P(v_i)}
=
P.
\] But $Q\subseteq K\subseteq P$. Therefore, $Q=K=P$ and $K$ is a polytope as claimed.
\end{proof}

\begin{proof}[Proof of \textup{(v)} implies \textup{(i)} in Theorem~\ref{thm1}]
    It suffices to combine Lemma~\ref{lem1} with Lemma~\ref{lem3}.
\end{proof}

{\bf Funding:} The first-named author was supported by the U.S. National Science Foundation's MSPRF fellowship via NSF grant DMS-2502744. 

{\bf Acknowledgments:} 
An LLM was used in the initial exploration of Conjecture~\ref{conj} and first produced a sketch of a proof of Conjecture \ref{conj}. Guided by this outline, the authors developed new ideas, which led to the more transparent proofs presented here.  We refuse to advertise for a particular private company and do not mention which LLM.

This project started at The Bremen Town Mathematicians (in Convexity) conference, and we thank the organizers and the University of Bremen for their hospitality.

\bibliographystyle{acm}
\bibliography{references}

@article{AS26,
  author  = {Averkov, Gennadiy and Soprunov, Ivan},
  title   = {An algebraic-combinatorial proof of a {B{\'e}zout}-type inequality for mixed volumes of three-dimensional zonoids},
  journal = {Discrete \& Computational Geometry},
  volume  = {76},
  pages = {653--663},
  year = {2026},
  doi = {10.1007/s00454-025-00745-2}
}

@article{SZ16,
  author = {Soprunov, I. and Zvavitch, A.},
  fjournal = {International Mathematics Research Notices. IMRN},
  issn = {1073-7928},
  journal = {Int. Math. Res. Not.},
  mrclass = {52A39 (52A40)},
  mrnumber = {3632081},
  number = {23},
  pages = {7230--7252},
  title = {Bezout inequality for mixed volumes},
  year = {2016}
}

@article{SSZ19,
  author = {Saroglou, Christos and Soprunov, Ivan and Zvavitch, Artem},
  title = {Wulff shapes and a characterization of simplices via a {B}ezout type inequality},
  journal = {Adv. Math.},
  fjournal = {Advances in Mathematics},
  volume = {357},
  year = {2019},
  pages = {106789, 24},
  issn = {0001-8708},
  mrclass = {52A39 (52A20 52A40 52B11)},
  mrnumber = {4017405},
  mrreviewer = {Uwe Schnell},
  doi = {10.1016/j.aim.2019.106789}
}

@article{HLPRY25,
  author = {Haddad, Juli{\'a}n and Langharst, Dylan and Putterman, Eli and Roysdon, Michael and Ye, Deping},
  title = {Affine isoperimetric inequalities for higher-order projection and centroid bodies},
  journal = {Math. Ann.},
  fjournal = {Mathematische Annalen},
  volume = {393},
  year = {2025},
  number = {1},
  pages = {1073--1121},
  issn = {0025-5831,1432-1807},
  mrclass = {52A39 (28A75 46E35 52A40)},
  mrnumber = {4966578},
  doi = {10.1007/s00208-025-03271-x},
  url = {https://doi.org/10.1007/s00208-025-03271-x}
}

@article{RS57,
  author = {Rogers, C. A. and Shephard, G. C.},
  title = {The difference body of a convex body},
  journal = {Arch. Math. (Basel)},
  fjournal = {Archiv der Mathematik},
  volume = {8},
  year = {1957},
  pages = {220--233},
  issn = {0003-889X},
  mrclass = {52.0X},
  mrnumber = {92172},
  mrreviewer = {W. Fenchel},
  doi = {10.1007/BF01899997}
}

@book{Sh1,
  author = {Schneider, Rolf},
  title = {Convex {B}odies: the {B}runn-{M}inkowski {T}heory},
  edition = {2nd expanded},
  series = {Encyclopedia of {M}athematics and its {A}pplications},
  volume = {151},
  year = {2014},
  publisher = {Cambridge University Press},
  address = {Cambridge, UK}
}

@article{SM24,
  author = {Szusterman, Maud},
  title = {Extremizers in {S}oprunov and {Z}vavitch's {B}ezout inequalities for mixed volumes},
  journal = {J. Math. Anal. Appl.},
  fjournal = {Journal of Mathematical Analysis and Applications},
  volume = {529},
  year = {2024},
  number = {2},
  pages = {Paper No. 127461, 23},
  issn = {0022-247X,1096-0813},
  mrclass = {52A39},
  mrnumber = {4650796},
  mrreviewer = {Qiang\ Tu},
  doi = {10.1016/j.jmaa.2023.127461},
  url = {https://doi.org/10.1016/j.jmaa.2023.127461}
}

@incollection{SM26,
  author = {Szusterman, Maud},
  title = {A new excluding condition towards the {S}oprunov-{Z}vavitch conjecture on {B{\'e}zout}-type inequalities},
  booktitle = {Geometry, analysis and convexity},
  series = {RSME Springer Ser.},
  volume = {17},
  pages = {25--39},
  publisher = {Springer, Cham},
  year = {2026},
  isbn = {978-3-032-11433-4; 978-3-032-11434-1},
  mrclass = {52A40 (49Q10)},
  mrnumber = {5072487},
  doi = {10.1007/978-3-032-11434-1_2},
  url = {https://doi.org/10.1007/978-3-032-11434-1_2}
}

@article{AFO14,
  author = {Artstein-Avidan, S. and Florentin, D. and Ostrover, Y.},
  doi = {10.1142/S0219199713500314},
  fjournal = {Communications in Contemporary Mathematics},
  issn = {0219-1997},
  journal = {Commun. Contemp. Math.},
  mrclass = {52A39 (52A40)},
  mrnumber = {3195153},
  mrreviewer = {Christos Saroglou},
  number = {2},
  pages = {1350031, 14},
  title = {Remarks about mixed discriminants and volumes},
  volume = {16},
  year = {2014}
}

@article{FMMZ24,
  author = {Fradelizi, Matthieu and Madiman, Mokshay and Meyer, Mathieu and Zvavitch, Artem},
  title = {On the volume of the {M}inkowski sum of zonoids},
  journal = {J. Funct. Anal.},
  fjournal = {Journal of Functional Analysis},
  volume = {286},
  year = {2024},
  number = {3},
  pages = {Paper No. 110247, 41},
  issn = {0022-1236,1096-0783},
  mrclass = {52A20},
  mrnumber = {4669585},
  doi = {10.1016/j.jfa.2023.110247},
  url = {https://doi.org/10.1016/j.jfa.2023.110247}
}

@article{Cha67,
  author = {Chakerian, G. D.},
  title = {Inequalities for the difference body of a convex body},
  journal = {Proc. Amer. Math. Soc.},
  fjournal = {Proceedings of the American Mathematical Society},
  volume = {18},
  year = {1967},
  pages = {879--884},
  issn = {0002-9939},
  mrclass = {52.40},
  mrnumber = {218972},
  mrreviewer = {G. C. Shephard},
  doi = {10.2307/2035131}
}

@misc{FHMNWZ26,
  author = {Fradelizi, Matthieu and Hubard, Alfredo and Manui, Auttawich and Ndiaye, Cheikh Saliou and Wang, Shouda and Zvavitch, Artem},
  title = {Volume and Projection Inequalities {I}: Zonoids and {Courtade}'s Conjecture},
  year = {2026},
  note = {arXiv:2608.12681},
  eprint = {2608.12681},
  eprinttype = {arXiv},
  eprintclass = {math.MG},
  doi = {10.48550/arXiv.2608.12681}
}

@incollection{GB23,
  author = {Bianchi, Gabriele},
  title = {The covariogram problem},
  booktitle = {Harmonic Analysis and Convexity},
  publisher = {De Gruyter},
  year = {2023},
  doi = {10.1515/9783110775389-002},
  pages = {37--82},
  series = {Adv. Anal. Geom.},
  editor = {Koldobsky, Alexander and Volberg, A.}
}

@article{SSZ16,
  author = {Saroglou, C. and Soprunov, I. and Zvavitch, A.},
  fjournal = {Proceedings of the American Mathematical Society},
  issn = {0002-9939},
  journal = {Proc. Amer. Math. Soc.},
  mrclass = {52A39 (52A40 52B11)},
  mrnumber = {3556275},
  mrreviewer = {P. R. Goodey},
  number = {12},
  pages = {5333--5340},
  title = {Characterization of simplices via the {B}ezout inequality for mixed volumes},
  volume = {144},
  year = {2016}
}

\end{document}